\documentclass{amsart}
\usepackage{amssymb}
\usepackage[all]{xy}

\newcommand{\Z}{\mathbb{Z}}

\newcommand{\C}{\mathbb{C}}
\newcommand{\Gm}{\mathbb{G}_{\mathrm{m}}}
\newcommand{\cU}{\mathcal{U}}
\newcommand{\cK}{\mathcal{K}}
\newcommand{\cF}{\mathcal{F}}
\newcommand{\cG}{\mathcal{G}}
\newcommand{\IC}{\mathrm{IC}}
\newcommand{\cS}{\mathcal{S}}
\DeclareMathOperator{\supp}{supp}

\newcommand{\cA}{\mathcal{A}}
\newcommand{\cB}{\mathcal{B}}

\newcommand{\cD}{\mathcal{D}}
\newcommand{\cP}{\mathcal{P}}
\newcommand{\cQ}{\mathcal{Q}}

\newcommand{\fh}{\mathfrak{h}}
\newcommand{\Sfh}{\mathrm{S}\fh}

\newcommand{\SW}{\mathbf{A}_W}
\newcommand{\lgmod}{\textup{-}\mathrm{gmod}}
\newcommand{\lgmodfd}{\textup{-}\mathrm{gmod}_{\mathrm{fd}}}
\newcommand{\Irr}{\mathrm{Irr}}

\newcommand{\bnabla}{{\bar\nabla}}
\newcommand{\bDelta}{{\bar\Delta}}

\newcommand{\tK}{{\tilde K}}
\newcommand{\Ex}{\mathcal{E}x_W}

\newcommand{\smush}{\mathbin{\#}}

\newcommand{\phy}{\mathbf{f}}

\newcommand{\Db}{D^{\mathrm{b}}}
\newcommand{\Dm}{D^-}
\newcommand{\Dp}{D^+}

\newcommand{\Dfd}{\Db_{\mathrm{fd}}}

\newcommand{\Hom}{\mathrm{Hom}}
\newcommand{\RHom}{\mathrm{Hom}^\bullet}
\newcommand{\Ext}{\mathrm{Ext}}
\newcommand{\End}{\mathrm{End}}
\newcommand{\uHom}{\underline{\Hom}}
\newcommand{\uExt}{\underline{\Ext}}
\newcommand{\uRHom}{\uHom^\bullet}

\newcommand{\op}{{\mathrm{op}}}
\newcommand{\D}{\mathbb{D}}

\newcommand{\la}{\langle}
\newcommand{\ra}{\rangle}
\newcommand{\simto}{\overset{\sim}{\to}}

\DeclareMathOperator{\gr}{gr}
\DeclareMathOperator{\im}{im}

\numberwithin{equation}{section}
\newtheorem{thm}{Theorem}[section]
\newtheorem{lem}[thm]{Lemma}
\newtheorem{prop}[thm]{Proposition}
\newtheorem{cor}[thm]{Corollary}

\theoremstyle{definition}
\newtheorem{defn}[thm]{Definition}

\theoremstyle{remark}
\newtheorem{rmk}[thm]{Remark}

\title{Kostka systems and exotic $t$-structures for reflection groups}

\author{Pramod N. Achar}
\address{Department of Mathematics\\
  Louisiana State University\\
  Baton Rouge, LA \ 70803\\
  U.S.A.}
\email{pramod@math.lsu.edu}

\thanks{The author received support from NSF Grant DMS-1001594.}

\begin{document}

\begin{abstract}
Let $W$ be a complex reflection group, acting on a complex vector space $\fh$.  Kato has recently introduced the notion of a ``Kostka system,'' which is a certain collection of finite-dimensional $W$-equivariant modules for the symmetric algebra on $\fh$.  In this paper, we show that Kostka systems can be used to construct ``exotic'' $t$-structures on the derived category of finite-dimensional modules, and we prove a derived-equivalence result for these $t$-structures.
\end{abstract}

\maketitle

\section{Introduction}

\subsection{Overview}

In the early 1980's, Shoji~\cite{sho:gpF4, sho:gpcg} and Lusztig~\cite{lus:cs5} showed that Green functions---certain  polynomials arising in the representation theory of finite groups of Lie type---can be computed by a rather elementary procedure, now often known as the \emph{Lusztig--Shoji algorithm}.  This algorithm can be interpreted as a computation in the Grothendieck group of the derived category of mixed $\ell$-adic complexes on the nilpotent cone of a reductive algebraic group, with the simple perverse sheaves playing a key role; see~\cite{a:gfhl}.

In recent work~\cite{kat:hsgp}, Kato has proposed an alternative interpretation of Green functions in terms of the Grothendieck group of the (derived) category of graded modules over the ring $\SW = \C[W] \smush \C[\fh^*]$, where $W$ is the Weyl group, and $\fh$ is the Cartan subalgebra.  In place of simple perverse sheaves, the key objects are now projective $\SW$-modules.  Thus, Kato's viewpoint is ``Koszul dual'' to the geometric one.  A prominent place is given to certain collections of finite-dimensional $\SW$-modules (denoted by $K_\chi$ in~\cite{kat:hsgp} and by $\bnabla_\chi$ here), called \emph{Kostka systems}.

In this paper, we study Kostka systems as generators of the derived category $\Dfd(\SW)$ of finite-dimensional $\SW$-modules.  We prove that they form a dualizable quasi-exceptional sequence, which implies that they determine a new $t$-structure on $\Dfd(\SW)$, called the \emph{exotic} $t$-structure.  The heart of this $t$-structure, denoted by $\Ex$, is a finite-length weakly quasi-hereditary category.  The main result (see Theorem~\ref{thm:dereq}) states that there is an equivalence of triangulated categories
\begin{equation}\label{eqn:mainthm}
\Db\Ex \simto \Dfd(\SW).
\end{equation}
Of course, projective $\SW$-modules cannot belong to $\Ex$, since they are not finite-dimensional.  Nevertheless, in some ways, they behave as though they were tilting objects of $\Ex$.  Thus, in a loose sense, which we do not attempt to make precise in this paper, the category $\Ex$ can be thought of ``Ringel dual'' to the category of $\SW$-modules.  (See Section~\ref{subsect:tilting}.)

\subsection{Analogy with geometric Langlands duality}

\begin{table}
\begin{center}
\setlength{\hangindent}{5pt}
\begin{tabular}{p{2.4in}|p{2.3in}}
\it geometric Langlands duality & \it Springer theory \\
\hline
perverse sheaves on the affine Grassmannian of $G$; geometric Satake &
perverse sheaves on the nilpotent cone of $G$; Springer correspondence \\
\smallskip $G \times \Gm$-equivariant coherent sheaves on the dual Lie algebra $\check{\mathfrak{g}}$ &
\smallskip graded $\SW$-modules, or $W \times \Gm$-equi\-variant coherent sheaves on $\check{\fh}$ \\
\smallskip coherent sheaves supported on the dual nilpotent cone $\check{\mathcal{N}} \subset \check{\mathfrak{g}}$ &
\smallskip finite-dimensional $\SW$-modules, or coherent sheaves supported on $\{0\} \subset \check{\fh}$ \\
\smallskip Andersen--Jantzen sheaves on $\check{\mathcal{N}}$ &
\smallskip Kostka systems $\{ \bnabla_\chi\}$ \\
\smallskip exotic (or perverse-coherent) $t$-structure on $\Db\mathrm{Coh}^{G \times \Gm}(\check{\mathcal{N}})$ &
\smallskip exotic $t$-struc\-ture on $\Dfd(\SW)$
\end{tabular}
\bigskip
\end{center}
\caption{Geometric Langlands duality and Springer theory}\label{tab:an}
\end{table}

A theme arising in geometric Langlands duality is that perverse or constructible sheaves on a (partial) affine flag variety for a reductive group $G$ should be described in terms of coherent sheaves on varieties related to the dual group $\check G$.  For instance, the spherical equivariant derived category of the affine Grassmannian $\mathsf{Gr}$ is closely related to coherent sheaves on the dual Lie algebra $\check{\mathfrak g}$; see~\cite{bf:esc}.

Springer theory is a rich source of phenomena that seem to be ``shadows at the level of the Weyl group'' of geometric Langlands duality.  Indeed, the Springer correspondence itself is in part a Weyl-group shadow of the geometric Satake equivalence~\cite{ah:gssc, ahr:gssc2}.  Another example is Rider's equivalence~\cite{rid:fnc} relating the equivariant derived category of the nilpotent cone to $\SW$-modules, or, equivalently, to $W$-equivariant coherent sheaves on the dual Cartan subalgebra $\check{\fh}$: this resembles the aforementioned result of~\cite{bf:esc}.  Further parallels are summarized in Table~\ref{tab:an}.

Kato's results and those of the present paper are contributions to the study of the ``Galois side'' (or ``coherent side'') of this picture.  Among (complexes of) coherent sheaves on $\check{\mathfrak g}$, those supported on the dual nilpotent cone $\check{\mathcal{N}}$ are of particular importance, especially those in the heart of an exotic $t$-structure determined by the so-called Andersen--Jantzen sheaves~\cite{bez:qes, bez:psaf}.  The Weyl-group analogue should involve sheaves supported on $\{0\} \subset \check{\fh}$---in other words, finite-dimensional $\SW$-modules.  Specifically, Kostka systems should be thought of as Weyl-group analogues of Andersen--Jantzen sheaves, and the equivalence~\eqref{eqn:mainthm} as a Weyl-group shadow of the derived equivalences from~\cite{bez:psaf} or~\cite[Theorem~1.2]{a:pcsnc}.

\subsection{Green functions for complex reflection groups}

The Lusztig--Shoji algorithm itself only requires knowing the reflection group $W$ and the preorder $\precsim$ on $\Irr(W)$ induced by the Springer correspondence.  (See~\cite{a:iglsa}.)  In particular, it makes sense to carry out the algorithm with a different, ``artificial'' preorder, or even with $W$ replaced by a complex reflection group that is not the Weyl group of any algebraic group.  See~\cite{sho:gf1, sho:gfls, gm:sp} for variations and conjectures on the Lusztig--Shoji algorithm.

One of Kato's aims in~\cite{kat:hsgp} was to provide a categorical framework for interpreting the output of the algorithm in this more general setting, where geometric tools like perverse sheaves are not available.  In the present paper, we try to preserve this goal.  Most definitions and constructions in this paper make sense for arbitrary complex reflection groups and arbitrary preorders on $\Irr(W)$.  We do invoke some results of Kato whose proofs involve the geometry of the nilpotent cone, and are thus valid only for Weyl groups.  However, outside of Section~\ref{sect:geom}, we treat these results as axioms: if, in the future, non-geometric proofs of these results become available for other complex reflection groups, then the main results of this paper will extend to those complex reflection groups as well.

\subsection{Acknowledgments}

I am grateful to Syu Kato for a number of helpful comments.  This paper has, of course, been deeply influenced by his ideas.  

\section{Notation and preliminaries}
\label{sect:prelim}

\subsection{Graded rings and vector spaces}

If $R$ is a noetherian graded $\C$-algebra, we write $R\lgmod$ (resp.~$R\lgmodfd$) for the category of finitely-generated (resp.~finite-dimensional) graded left $R$-modules.  For any $M \in R\lgmod$, we write $\gr_k V$ for its $k$-th graded component.  We define $M\la 1\ra$ to be the new graded module with
\[
\gr_k (M\la 1\ra) = \gr_{k-1} M.
\]
The operation $M \mapsto M\la 1\ra$ also makes sense for chain complexes of modules over $R$.  If $M$ and $N$ are (complexes of) graded $R$-modules, we define  $\uHom_R(M,N)$ (or simply $\uHom(M,N)$) to be the graded vector space given by
\[
\gr_k \uHom_R(M,N) = \Hom(M, N\la -k\ra).
\]

We use the term \emph{grade} to refer to the integers $k$ such that $\gr_k M \ne 0$, reserving the term \emph{degree} for homological uses, such as indexing the terms in a chain complex.  Thus, a module $M$ is said to \emph{have grades${}\ge n$} if $\gr_k M = 0$ for all $k < n$.  If $M$ is a chain complex of modules, we say that $M$ \emph{has grades${}\ge n$} if all its cohomology modules $H^i(M)$ have grades${}\ge n$.

If $M$ and $N$ are objects in a derived category of $R$-modules, we employ the usual notation
$\Hom^i(M,N) = \Hom(M, N[i])$, as well as $\uHom^i(M,N) = \uHom(M, N[i])$.

\subsection{Reflection groups and phyla}

Throughout the paper, $W$ will be a fixed complex reflection group, acting on a finite-dimensional complex vector space $\fh$.  Let $\Sfh$ be the symmetric algebra on $\fh$, regarded as a graded ring by declaring elements of $\fh \subset \Sfh$ to have degree~$1$.   Our main object of study is the ring
\[
\SW = \C[W] \smush \Sfh.
\]
Let $\SW\lgmod$ be the category of finitely-generated graded $\SW$-modules.  Henceforth, all $\SW$-modules are assumed to be objects of $\SW\lgmod$.

Let $\Irr(W)$ denote the set of irreducible complex characters of $W$.    For $\chi \in \Irr(W)$, let $\bar\chi$ denote the complex-conjugate character.  If $W$ is a Coxeter group, then all characters are real-valued, and $\bar\chi = \chi$, but general complex reflection groups may have characters that are not real-valued.

We also assume throughout that $\Irr(W)$ is equipped with a fixed total preorder $\precsim$, and that the equivalence relation $\sim$ induced by this preorder satisfies
\[
\chi \sim \bar\chi
\]
for all $\chi \in \Irr(W)$.  (In~\cite{kat:hsgp}, a preorder satisfying this condition is said to be \emph{of Malle type}.  Many arguments in this paper can likely be adapted to the case where this condition is dropped, but these generalizations will not be pursued here.)  Following~\cite{a:iglsa}, the equivalence classes for $\sim$ are called \emph{phyla}.  For $\chi \in \Irr(W)$, we write $[\chi]$ for the phylum to which it belongs.

\subsection{$\SW$-modules}

For each $\chi \in \Irr(W)$, choose a representation $L_\chi$ giving rise to that character.  Consider the vector space
\[
P_\chi = L_\chi \otimes \Sfh.
\]
We regard this as a graded $\SW$-module by having $\Sfh$ act on the second factor, and having $W$ act on both factors.  This is a projective $\SW$-module, and every indecomposable projective in $\SW\lgmod$ is of the form $P_\chi\la n\ra$ for some $\chi$ and some~$n$.  See~\cite[Lemma~2.2]{kat:hsgp}.

For brevity, we write $\Db(\SW)$ rather than $\Db(\SW\lgmod)$ for the bounded derived category of $\SW\lgmod$, and likewise for $\Dm(\SW)$ and $\Dp(\SW)$.

We will occasionally need to consider groups
\begin{equation}\label{eqn:hom-unbdd}
\Hom(M,N) \qquad\text{with $M \in \Dp(\SW)$ and $N \in \Dm(\SW)$.}
\end{equation}
This is to be understood by identifying $\Dp(\SW)$ and $\Dm(\SW)$ with full subcategories of the unbounded derived category $D(\SW)$.  Because $\SW$ has finite global dimension, we can ignore some of the technical difficulties that usually arise with unbounded derived categories.  In particular, according to~\cite[Proposition~3.4]{af:hduc}, complexes of projective modules in $D(\SW)$ are homotopy-projective.  Moreover, every object in $\Dp(\SW)$ is isomorphic to a bounded-below complex of projectives; see~\cite[\S 1.6]{af:hduc}.  Thus, if $M$ and $N$ are both given by explicit complexes of projectives, then~\eqref{eqn:hom-unbdd} is simply the set of homotopy classes of chain maps between those complexes.

\subsection{Duality}

For $M \in \Sfh\lgmod$, the graded vector space $\uHom_{\Sfh}(M,\Sfh)$ can naturally be regarded as an object of $\Sfh\lgmod$ itself.  It is well known that the derived functor $\D = R\uHom_{\Sfh}({-}, \Sfh)$ given an equivalence of categories $\Dm(\Sfh)^\op \simto \Dp(\Sfh)$; see~\cite[Example~V.2.2]{har:rd}.  Moreover, $\D$ takes bounded complexes to bounded complexes, and so gives an antiautoequivalence of $\Db(\Sfh)$.

Now, suppose that $M \in \SW\lgmod$.  Then the $\Sfh$-module $\uHom_{\Sfh}(M,\Sfh)$ carries an obvious $W$-action, and so can be regarded as an object of $\SW\lgmod$.  From the facts above about $\D$, one can deduce the $W$-equivariant analogues: there is an equivalence of categories
\[
\D = R\uHom_{\Sfh}({-},\Sfh): \Dm(\SW)^\op \simto \Dp(\SW)
\]
that restricts to an equivalence $\Db(\SW)^\op \simto \Db(\SW)$.  In particular, we have
\begin{equation}\label{eqn:dual-free}
\D(P_\chi) \cong P_{\bar\chi}.
\end{equation}

\subsection{Finite-dimensional modules}

As noted in the introduction, the main results of this paper involve the category
\[
\Dfd(\SW) = \left\{ X \in \Db(\SW) \;\Big|\,
\begin{array}{c}
\text{for all $i$, $H^i(X)$ is a}\\
\text{finite-dimensional $\SW$-module}
\end{array}
\right\}.
\]
We will occasionally make use of the fact that this is equivalent to the derived category $\Db(\SW\lgmodfd)$.  That fact is an instance of the following lemma.

\begin{lem}\label{lem:fd}
Let $R = \bigoplus_{n \ge 0} R_n$ be a nonnegatively graded noetherian $\C$-algebra, and assume that $R_0$ is finite-dimensional.  Then the natural functor
\[
\Db(R\lgmodfd) \to \Db(R\lgmod)
\]
is fully faithful.
\end{lem}
\begin{proof}
We begin with a digression.  Since $R$ is noetherian and $R_0$ is finite-dimen\-sion\-al, each $R_n$ must be finite-dimensional.  It follows that for any $M \in R\lgmod$, each $\gr_n M$ is finite-dimensional.  Now, given $k \in \Z$, let $M_{\ge k} \subset M$ be the submodule generated by all homogeneous elements of grade${}\ge k$, and let $M_{\le k} = M/M_{\ge k+1}$.  It is easy to see that the functors $M \mapsto M_{\ge k}$ and $M \mapsto M_{\le k}$ are exact.  Moreover, $M_{\le k}$ is always finite-dimensional.

Returning to the statement of the lemma, recall that by a standard argument (see~\cite[Proposition~3.1.16]{bbd}), the question can be reduced to showing that the following natural morphism of $\delta$-functors (for $A, B \in R\lgmodfd$) is an isomorphism:
\begin{equation}\label{eqn:extfd}
\Ext^i_{R\lgmodfd}(A,B) \to \Ext^i_{R\lgmod}(A,B).
\end{equation}
When $i = 0$, this is obvious, and for $i = 1$, this follows from the fact that $R\lgmodfd$ is a Serre subcategory of $R\lgmod$.

For general $i>0$, each element of $\Ext^i_{R\lgmod}(A,B)$ is represented by some exact sequence
\begin{equation}\label{eqn:yoneda}
0 \to B \to M^i \to M^{i-1} \to \cdots \to M^1 \to A \to 0.
\end{equation}
Since $A$ and $B$ are finite-dimensional, there is a $k$ such that $A_{\ge k+1} = B_{\ge k+1} = 0$.  Applying the exact functor $M \mapsto M_{\le k}$ to~\eqref{eqn:yoneda} gives an exact sequence
\begin{equation}\label{eqn:yoneda2}
0 \to B \to M^i_{\le k} \to M^{i-1}_{\le k} \to \cdots \to M^1_{\le k} \to A \to 0.
\end{equation}
This represents the same element of $\Ext^i_{R\lgmod}(A,B)$ as~\eqref{eqn:yoneda}, but since every term is finite-dimensional, it also represents an element of $\Ext^i_{R\lgmodfd}(A,B)$.  We have just shown that~\eqref{eqn:extfd} is surjective for all $i$.

According to~\cite[Remarque~3.1.17(1)]{bbd}, if~\eqref{eqn:extfd} failed to be an isomorphism for some $i$, then for a minimal such $i$, it would be injective but not surjective.  So~\eqref{eqn:extfd} is indeed an isomorphism for all $i$.
\end{proof}

\subsection{Admissible subcategories of triangulated categories}

We conclude this section with a review of a result from homological algebra that we will use a number of times in the sequel.  

\begin{defn}
Let $\cD$ be a triangulated category, and let $\cA$ and $\cB$ be two full triangulated subcategories.  We say that $(\cA, \cB)$ is an \emph{admissible pair} if the following two conditions hold:
\begin{enumerate}
\item We have $\Hom(A,B) = 0$ whenever $A \in \cA$ and $B \in \cB$.
\item Together, the objects in $\cA$ and $\cB$ generate $\cD$ as a triangulated category.
\end{enumerate}
\end{defn}

This is slightly nonstandard terminology: usually, $\cA$ is said to be \emph{right-admis\-sible} if there exists a $\cB$ such that the conditions above hold; dually, $\cB$ is said to be \emph{left-admissible}.  The following lemma collects some consequences and equivalent characterizations.

\begin{lem}[{\cite[Propositions~1.5 and 1.6]{bk:rfsfr}}]\label{lem:admissible}
Let $(\cA, \cB)$ be an admissible pair in a triangulated category $\cD$.   Then:
\begin{enumerate}
\item The inclusion $\cA \to \cD$ admits a right adjoint $\imath: \cD \to \cA$.\label{it:rt-adjoint}
\item The inclusion $\cB \to \cD$ admits a left adjoint $\jmath: \cD \to \cB$.\label{it:lt-adjoint}
\item For every $X \in \cD$, there is a functorial distinguished triangle\label{it:adm-dt}
\[
\imath(X) \to X \to \jmath(X) \to.
\]
\item We have $\cA = \{X \in \cD \mid \text{$\Hom(X,B) = 0$ for all $B \in \cB$} \}$.\label{it:rt-orth}
\item We have $\cB = \{X \in \cD \mid \text{$\Hom(A,X) = 0$ for all $A \in \cA$} \}$.\label{it:lt-orth}
\item The inclusions $\cA \to \cD$ and $\cB \to \cD$ induce equivalences of triangulated categories\label{it:adm-quot}
\[
\cA \simto \cD/\cB
\qquad\text{and}\qquad \cB \simto \cD/\cA.
\]
\end{enumerate}
\end{lem}

Note, in particular, that each of $\cA$ and $\cB$ determines the other.

\section{Triangulated subcategories associated to a phylum}
\label{sect:phyla}

Given a phylum $\phy$, we define a full subcategory of $\Dm(\SW)$ as follows:
\[
\Dm(\SW)_{\preceq \phy} =
\left\{ X \in \Dm(\SW) \;\Bigg|\,
\begin{array}{c}
\text{$X$ is isomorphic to a bounded-above} \\
\text{complex $M^\bullet$ where each $M^i$ is a direct} \\
\text{sum of various $P_\chi\la n\ra$ with $[\chi] \preceq \phy$}
\end{array}
\right\}.
\]
We will also consider the ``strict'' version $\Dm(\SW)_{\prec \phy}$, as well as the analogous subcategories of $\Db(\SW)$ and $\Dp(\SW)$.  It follows from~\eqref{eqn:dual-free} that
\begin{equation}\label{eqn:dual-strat}
\D(\Dm(\SW)_{\preceq \phy}) = \Dp(\SW)_{\preceq \phy}
\qquad\text{and}\qquad
\D(\Db(\SW)_{\preceq \phy}) = \Db(\SW)_{\preceq \phy}.
\end{equation}
In addition, we have
\[
\Db(\SW)_{\preceq \phy} = \Dm(\SW)_{\preceq \phy} \cap \Db(\SW) = \Dp(\SW)_{\preceq \phy} \cap \Db(\SW).
\]
The first of these holds by a routine homological-algebra argument for bounded-above complexes of projectives over a ring with finite global dimension.  The second equality follows from the first using~\eqref{eqn:dual-strat}.

In this section, we first construct a collection of objects in $\Dm(\SW)$ and $\Dp(\SW)$ with various $\Hom$-vanishing properties related to the categories defined above.  Then, under the additional assumption that these objects lie in $\Db(\SW)$, we prove structural results for that category in the spirit of Lemma~\ref{lem:admissible}.

\subsection{Construction of $\nabla_\chi$ and $\Delta_\chi$}

We begin with the following result.

\begin{prop}\label{prop:nabla-exist}
For each $\chi \in \Irr(W)$, there is an object $\nabla_\chi \in \Dm(\SW)$ together with a morphism $s: P_\chi \to \nabla_\chi$ with the following properties:
\begin{enumerate}
\item The cone of $s$ lies in $\Dm(\SW)_{\prec [\chi]}$.\label{it:nabla-cone}
\item For $M \in \Dm(\SW)_{\prec [\chi]}$ or $\Dp(\SW)_{\prec [\chi]}$, we have $\Hom(M, \nabla_\chi) = 0$.\label{it:nabla-orth}
\end{enumerate}
Moreover, the pair $(\nabla_\chi, s)$ is unique up to unique isomorphism.
\end{prop}
\begin{proof}
Given a module $M \in \SW\lgmod$, let $M_{\prec [\chi]}$ be the $\SW$-submodule generated by all the homogeneous $W$-stable subspaces that are isomorphic to some $L_\psi\la m\ra$ with $\psi \prec \chi$.  Of course, $M_{\prec [\chi]}$ is actually generated by a finite number of such subspaces.  Thus, there is a surjective map $M' \twoheadrightarrow M_{\prec [\chi]}$, where $M'$ is a direct sum of finitely many objects of the form $P_\psi\la n\ra$ with $\psi \prec \chi$.

We now define a complex $(N^\bullet,d^\bullet)$ inductively as follows.  Let $N^i = 0$ for $i > 0$, and let $N^0 = P_\chi$.  Then, assuming that $N^i$ and $d^i: N^i \to N^{i+1}$ have already been defined for $i > j$, let us apply the construction of the preceding paragraph to $M = \ker d^{j+1} \subset N^{j+1}$.  Set $N^j = M'$, and then let $d^j: N^j \to N^{j+1}$ be the composition
\[
N^j \twoheadrightarrow (\ker d^{j+1})_{\prec [\chi]} \hookrightarrow N^{j+1}.
\]
Let $\nabla_\chi = (N^\bullet, d^\bullet)$.  There is an obvious morphism $s: P_\chi \to \nabla_\chi$.  Its cone is isomorphic to the complex obtained from $(N^\bullet, d^\bullet)$ by omitting $N^0$.  By construction, the $N^i$ for $i < 0$ are direct sums of $P_\psi\la n\ra$ with $\psi \prec \chi$, so it is clear that the cone of $s$ lies in $\Dm(\SW)_{\prec [\chi]}$.

For $M \in \Dm(\SW)_{\prec [\chi]}$ given by a suitable bounded-above complex of projectives, it is a routine exercise in homological algebra to show that any map $M \to \nabla_\chi$ is null-homotopic.  On the other hand, if $M \in \Dp(\SW)_{\prec [\chi]}$ is given by a bounded-below complex of projectives, let $M'$ be the subcomplex obtained by omitting the terms in degrees${}\le 1$, and form a distinguished triangle $M' \to M \to M'' \to$.  Then $M''$ lies in $\Db(\SW)_{\prec [\chi]}$.  It is clear that $\Hom(M', \nabla_\chi) = \Hom(M'[1],\nabla_\chi) = 0$, and thus $\Hom(M, \nabla_\chi) = \Hom(M'',\nabla_\chi) = 0$ as well.

Finally, suppose $s': P_\chi \to \nabla'_\chi$ were another morphism with the same properties, and let $C'$ be its cone.  Since $\Hom(C[-1],\nabla'_\chi) = 0$, the map $s'$ factors through $s$, and then the last assertion follows by a standard argument.
\end{proof}

\begin{rmk}\label{rmk:nabla-exist}
In the construction above, it is easy to see by induction that the complex $(N^\bullet, d^\bullet)$ representing $\nabla_\chi$ can be chosen such that each nonzero $N^j$ is generated in grades${}\ge -j \ge 0$.  It follows that $\nabla_\chi$ has grades${}\ge 0$.
\end{rmk}

\begin{prop}\label{prop:delta-exist}
For each $\chi \in \Irr(W)$, there is an object $\Delta_\chi \in \Dp(\SW)$ together with a morphism $t: \Delta_\chi \to P_\chi$ with the following properties:
\begin{enumerate}
\item The cone of $t$ lies in $\Dp(\SW)_{\prec [\chi]}$.\label{it:delta-cone}
\item For $M \in \Dm(\SW)_{\prec [\chi]}$ or $\Dp(\SW)_{\prec [\chi]}$, we have $\Hom(\Delta_\chi, M) = 0$.\label{it:delta-orth}
\end{enumerate}
Moreover, the pair $(\Delta_\chi, t)$ is unique up to unique isomorphism.
\end{prop}
\begin{proof}
Let $\Delta_\chi = \D(\nabla_{\bar\chi})$, and let $t = \D(s): \Delta_\chi \to P_\chi$.  It follows from~\eqref{eqn:dual-free}, \eqref{eqn:dual-strat}, and Proposition~\ref{prop:nabla-exist} that $(\Delta_\chi, t)$ has the required properties.
\end{proof}

\begin{cor}\label{cor:delta-nabla-orth}
\begin{enumerate}
\item If $\chi \not\sim \psi$, then $\uRHom(\Delta_\chi, \nabla_\psi) = 0$.\label{it:orth-diff}
\item If $i > 0$, then $\uHom^i(\Delta_\chi, \nabla_\psi) = 0$ for all $\chi, \psi$.\label{it:orth-ext}
\end{enumerate}
\end{cor}
\begin{proof}
The first assertion follows from Propositions~\ref{prop:nabla-exist}\eqref{it:nabla-orth} and~\ref{prop:delta-exist}\eqref{it:nabla-orth}.  For the second, observe that by construction, $\nabla_\psi$ is isomorphic to a complex of projectives in nonpositive degrees, so $\Delta_\psi$ is isomorphic to a complex of projectives in nonnegative degrees.  The result then follows by the remarks after~\eqref{eqn:hom-unbdd}.
\end{proof}

\subsection{Admissible subcategories of $\Db(\SW)$}

For the remainder of this section, we impose the additional assumption that all the $\Delta_\chi$ and $\nabla_\chi$ lie in $\Db(\SW)$.  With this assumption, it makes sense to consider the following full triangulated subcategories of $\Db(\SW)$:
\begin{align*}
\Db(\SW)_{\phy} &= \text{the triangulated subcategory generated by the $\nabla_\chi\la n\ra$ with $\chi \in \phy$,} \\
\Db(\SW)^{\phy} &= \text{the triangulated subcategory generated by the $\Delta_\chi\la n\ra$ with $\chi \in \phy$.}
\end{align*}
We will see below that these two categories are equivalent.  It often happens that the $\nabla_\chi$ are easier to work with explicitly than the $\Delta_\chi$, so this equivalence will be useful for transfering facts about the former to the setting of the latter.

\begin{prop}\label{prop:nabla-gen}
For each phylum $\phy$, $\Db(\SW)_{\preceq \phy}$ is generated as a triangulated category by the $\nabla_\chi\la n\ra$ (resp.~the $\Delta_\chi\la n\ra$) with $[\chi] \preceq \phy$.
\end{prop}
\begin{proof}
This follows by induction on $\phy$ with respect to the order on the set of phyla, using the distinguished triangle $P_\chi \to \nabla_\chi \to C \to $ with $C \in \Db(\SW)_{\prec \phy}$.
\end{proof}

In the case of the $\nabla_\chi$, this statement can be refined a bit.  Recall from Remark~\ref{rmk:nabla-exist} that $\nabla_\chi$ has grades${}\ge 0$.  It follows that in the distinguished triangle $P_\chi \to \nabla_\chi \to C\to$, the object $C$ has grades${}\ge 0$.  By keeping track of grades in the induction, one can see that each $P_\psi$ is contained in the triangulated category generated by the $\nabla_\chi\la k\ra$ with $k \ge 0$.  We have just shown that part~\eqref{it:grades-proj} in the corollary below implies part~\eqref{it:grades-nab}.  (Note, in contrast, that the $\Delta_\chi$ do not, in general, have grades${}\ge 0$.)

\begin{cor}\label{cor:grades}
The following conditions on an object $M \in \Db(\SW)$ are equivalent:
\begin{enumerate}
\item $M$ has grades${}\ge n$.\label{it:grades-reg}
\item $M$ is isomorphic to a complex of projective $\SW$-modules each term of which has grades${}\ge n$.\label{it:grades-proj}
\item $M$ lies in the triangulated subcategory generated by the $\nabla_\chi\la k \ra$ with $k \ge n$.\label{it:grades-nab}
\end{enumerate}
\end{cor}
\begin{proof}
We saw above that~\eqref{it:grades-proj} implies~\eqref{it:grades-nab}. It is a routine exercise to see that~\eqref{it:grades-reg} implies~\eqref{it:grades-proj}, and Remark~\ref{rmk:nabla-exist} tells us that~\eqref{it:grades-nab} implies~\eqref{it:grades-reg}.
\end{proof}

\begin{cor}\label{cor:phy-admis}
Each of the two pairs of categories $(\Db(\SW)_{\phy}, \Db(\SW)_{\prec \phy})$ and  $(\Db(\SW)_{\prec \phy}, \Db(\SW)^{\phy})$ is an admissible pair in $\Db(\SW)_{\preceq \phy}$.
\end{cor}
\begin{proof}
This follows from Propositions~\ref{prop:nabla-exist}\eqref{it:nabla-orth}, \ref{prop:delta-exist}\eqref{it:delta-orth}, and~\ref{prop:nabla-gen}.
\end{proof}

The next two results are just restatements of parts~\eqref{it:rt-orth}--\eqref{it:adm-quot} of Lemma~\ref{lem:admissible}.

\begin{prop}\label{prop:nabla-admis}
Let $\phy$ be a phylum, and let $M \in \Db(\SW)$. The following three conditions are equivalent:
\begin{enumerate}
\item $M \in \Db(\SW)_{\prec \phy}$.
\item $\uRHom(M, \nabla_\chi) = 0$ for all $\chi$ with $[\chi] \succeq \phy$.
\item $\uRHom(\Delta_\chi, M) = 0$ for all $\chi$ with $[\chi] \succeq \phy$. \qed
\end{enumerate}
\end{prop}

\begin{lem}\label{lem:quotient}
The inclusion functors $\Db(\SW)_{\phy} \to \Db(\SW)$ and $\Db(\SW)^{\phy} \to \Db(\SW)$ induce equivalences of categories
\[
\Db(\SW)_{\phy} \simto \Db(\SW)_{\preceq \phy}/ \Db(\SW)_{\prec \phy} \overset{\sim}{\leftarrow} \Db(\SW)^{\phy}.\qed
\]
\end{lem}

Let us denote the composition of these two equivalences by
\begin{equation}\label{eqn:t-defn}
T_\phy: \Db(\SW)_{\phy} \simto \Db(\SW)^{\phy}.
\end{equation}

\begin{prop}\label{prop:delta-nabla-cone}
For each $\chi$, there is a morphism $i: \Delta_\chi \to \nabla_\chi$ whose cone lies in $\Db(\SW)_{\prec [\chi]}$.  As a consequence, we have $T_{[\chi]}(\nabla_\chi) \cong \Delta_\chi$.
\end{prop}
\begin{proof}
Consider the distinguished triangles $P_\chi \overset{s}{\to} \nabla_\chi \to C \to$ and and $\Delta_\chi \overset{t}{\to} P_\chi \to C' \to $.  Let $i = s \circ t$, and let $K$ be its cone.  Applying the octahedral axiom to this composition, one finds that there is a distinguished triangle of the form $C' \to K \to C \to$, and thus $K \in \Db(\SW)_{\prec [\chi]}$.
\end{proof}

The preceding proposition says that $\Delta_\chi$ and $\nabla_\chi$ become isomorphic in the quotient $\Db(\SW)_{\preceq \phy}/\Db(\SW)_{\prec \phy}$.  Following this isomorphism through the equivalences of Lemma~\ref{lem:quotient} gives us the next result.

\begin{cor}\label{cor:delta-nabla-same}
If $\chi \sim \psi$, then we have natural isomorphisms
\[
\uRHom(\Delta_\chi, \Delta_\psi) \simto \uRHom(\Delta_\chi, \nabla_\psi) \overset{\sim}{\leftarrow} \uRHom(\nabla_\chi, \nabla_\psi). \qed
\]
\end{cor}

\begin{cor}\label{cor:nabla-ext}
If $\chi \sim \psi$, then $\uHom^i(\Delta_\chi, \Delta_\psi) = \uHom^i(\nabla_\chi, \nabla_\psi) = 0$ for $i > 0$.
\end{cor}
\begin{proof}
This follows from Corollary~\ref{cor:delta-nabla-orth}\eqref{it:orth-ext} and Corollary~\ref{cor:delta-nabla-same}.
\end{proof}

\subsection{Negative $\Ext$-vanishing for the $\nabla_\chi$}

It was remarked earlier that the $\nabla_\chi$ are often easier to work with than the $\Delta_\chi$.  The reason is that the $\nabla_\chi$ often belong to $\SW\lgmod$.  The following proposition gives a criterion for this to hold.  This proposition will not be used elsewhere in the paper, since, in the context of generalized Springer correspondences, Kato has shown this using a rather different argument.  The assumption that $\nabla_\chi \in \Db(\SW)$ remains in force.

\begin{prop}\label{prop:costd-abel}
The following conditions are equivalent:
\begin{enumerate}
\item We have $\uHom^i(\nabla_\chi, \nabla_\psi) = 0$ for all $i < 0$ and all $\chi, \psi \in \Irr(W)$.
\item We have $\nabla_\chi \in \SW\lgmod$ for all $\chi \in \Irr(W)$.
\end{enumerate}
\end{prop}
\begin{proof}
It is obvious that the second condition implies the first, so we will focus on the other implication.  For an object $X$, let $[X]$ denote its isomorphism class.  We will make use of the ``$*$'' operation for triangulated categories; see~\cite[\S 1.3.9]{bbd}.  If $\mathcal{X}$ and $\mathcal{Y}$ are two sets of isomorphism classes of objects, then $\mathcal{X} * \mathcal{Y}$ is the set of isomorphism classes $[Z]$ such that $Z$ fits into a distinguished triangle $X \to Z \to Y \to$ with $[X] \in \mathcal{X}$ and $[Y] \in \mathcal{Y}$.  This operation is associative.  

We claim that there are characters $\psi_i \in \Irr(W)$ and integers $n_i$, $k_i$ such that
\begin{equation}\label{eqn:l-nabla-star}
[P_\chi] \in \{[\nabla_{\psi_1}\la n_1\ra[k_1]]\} * \{[\nabla_{\psi_2}\la n_2\ra[k_2]]\} * \cdots * \{[\nabla_{\psi_j}\la n_j\ra[k_j]]\} * \{[\nabla_\chi]\}
\end{equation}
and where $\psi_i \prec \chi$ and $k_i \ge 0$ for all $i$.  We prove this claim by induction with respect to the preorder $\precsim$.  If $\chi$ is minimal, then $P_\chi = \nabla_\chi$, and there is nothing to prove.  Otherwise, form the distinguished triangle $P_\chi \to \nabla_\chi \to C \to$, so that \begin{equation}\label{eqn:l-c-nabla}
[P_\chi] \in \{[C[-1]]\} * \{[\nabla_\chi]\}.
\end{equation}
Referring to the explicit construction in the proof of Proposition~\ref{prop:nabla-exist} again, we see that $C[-1]$ is given by a complex of $P_\theta\la m\ra$'s concentrated in degrees${}\le 0$, with $\theta \prec \chi$.  In other words, there is an expression of the form
\[
[C[-1]] \in \{[P_{\theta_1}\la m_1\ra[p_1]]\} * \cdots * \{[P_{\theta_k}\la m_k\ra[p_k]]\},
\]
where $\theta_i \prec \chi$ and $p_i \ge 0$ for all $i$.  By induction, we can replace each term here by one of the form~\eqref{eqn:l-nabla-star}.  Combining this with~\eqref{eqn:l-c-nabla} yields the desired expression for $P_\chi$, so the proof of~\eqref{eqn:l-nabla-star} is complete.

We now claim that
\[
\uHom^i(P_\chi, \nabla_\psi\la n\ra) = 0 \qquad\text{if $i \ne  0$.}
\]
Indeed, for $i > 0$, this is obvious by construction, whereas for $i < 0$, it follows from~\eqref{eqn:l-nabla-star} and the assumption that $\uHom^j(\nabla_\theta, \nabla_\psi)$ vanishes for $j < 0$.  Finally, we observe that an object $X$ of $\Db(\SW)$ lies in $\SW\lgmod$ if and only if $\uHom^i(P_\chi,X) = 0$ for all $\chi$ and all $i \ne 0$.
\end{proof}

\section{Results from the geometry of generalized Springer correspondences}
\label{sect:geom}

One source of natural preorders on $\Irr(W)$ for certain Coxeter groups $W$ is Lusztig's \emph{generalized Springer correspondence}~\cite{lus:icc}, which involves the study of certain perverse sheaves on the unipotent variety of a reductive group.  In this setting, Kato has shown~\cite{kat:hsgp, kat:asea} that one can exploit the geometry to prove a number of strong results about the $P_\chi$ and the $\nabla_\chi$.  His results are stated in a somewhat different language, however, so this section is devoted to rephrasing Kato's results in terms that are better suited to the aims of the present paper.  In particular, we prove that the $\nabla_\chi$ coincide with the modules  denoted $\tK_\chi$ in~\cite{kat:hsgp}.  The argument given here is an adaptation of one given by Kato in~\cite{kat:asea}.

We begin with some notation.  Let $G$ be a connected complex reductive algebraic group, and let $\cU$ denote its unipotent variety.  Let $L \subset G$ be a Levi subgroup, $C_1 \subset L$ a unipotent class in $L$, and $E_1$ a local system on $C_1$ such that the triple $(L,C_1,E_1)$ appears in~\cite[Theorem~6.5]{lus:icc}.  According to~\cite[Theorem~9.2]{lus:icc}, the group $N_G(L)/L$, where $N_G(L)$ is the normalizer of $L$ in $G$, is a Coxeter group.

Let $W = N_G(L)/L$.  Associated to the data $(L, C_1, E_1)$ is a certain semisimple $G$-equivariant perverse sheaf $\cK$ on $\cU$ that is equipped with a natural isomorphism $\End(\cK) \cong \C[W]$.  This isomorphism determines a decomposition
\[
\cK \cong \bigoplus_{\chi \in \Irr(W)} \IC_\chi \otimes L_\chi,
\]
where the $\IC_\chi$ are distinct simple perverse sheaves.  Define a preorder on $\Irr(W)$ by
\begin{equation}\label{eqn:geom-preorder}
\chi \precsim \psi \qquad\text{if}\qquad \supp \IC_\chi \subset \overline{\supp \IC_\psi}.
\end{equation}
The support of each $\IC_\chi$ is the closure of one unipotent class, so the phyla of this preorder can be identified with a subset of the set of unipotent classes of $G$.

Throughout this section, we will assume that $W$ and $\precsim$ arise in this way.  We will use the equivariant derived category $\Db_G(\cU)$ to construct certain $\SW$-modules.  The starting point is the fact~\cite{lus:cls2,kat:hsgp} that there is an isomorphism of graded rings
\[
\bigoplus_{i \ge 0} \Hom^{2i}_{\Db_G(\cU)}(\cK,\cK) \cong \SW.
\]
(Also, $\Hom^i(\cK,\cK) = 0$ for $i$ odd.)
This lets us define an additive functor
\[
\cS: \Db_G(\cU) \to \SW\lgmod
\qquad\text{by}\qquad
\gr_k \cS(\cF) = \Hom^{2k}_{\Db_G(\cU)}(\cK,\cF).
\]
In particular, we have $\cS(\IC_\chi) \cong P_\chi$.  Moreover, for any $\cF$, the natural map
\[
\Hom(\cK,\cF[2k]) \simto \Hom(\cS(\cK), \cS(\cF)\la -2k\ra)
\]
is an isomorphism, as both sides are naturally identified with $\gr_k \cS(\cF)$.  Since each $\IC_\chi$ is a direct summand of $\cK$, it follows that the natural map
\begin{equation}\label{eqn:cs-ff}
\Hom(\IC_\chi, \cF[2k]) \simto \Hom(P_\chi, \cS(\cF)\la -2k\ra)
\end{equation}
is also an isomorphism.

The next two lemmas involve the following notion from~\cite{jmw}: an object $\cF \in \Db_G(\cU)$ is said to be \emph{$*$-even} (resp.~\emph{$!$-even}) if for each unipotent class $j_C: C \hookrightarrow \cU$, the cohomology sheaves $\mathcal{H}^k(j_C^*\cF)$ (resp.~$\mathcal{H}^k(j_C^!\cF)$) vanish whenever $k$ is odd.

\begin{lem}
Let $\cF' \to \cF \to \cF'' \to$ be a distinguished triangle of $!$-even objects in $\Db_G(\cU)$.  Then the sequence
$
0 \to \cS(\cF') \to \cS(\cF) \to \cS(\cF'') \to 0
$
is exact.
\end{lem}
\begin{proof}
According to~\cite[Theorem~24.8(a)]{lus:cs5}, the object $\cK$ is $*$-even.  As explained in~\cite[Remark~2.7]{jmw}, if $\cG$ is $!$-even, then $\Hom^k(\cK,\cG) = 0$ when $k$ is odd.  Thus, in the long exact sequence obtained by applying $\Hom(\cK,{-})$ to the given distinguished triangle, all the odd-degree terms vanish, and the result follows.
\end{proof}

\begin{lem}\label{lem:dguk}
Let $\phy$ be a phylum, and let $C_\phy \subset \cU$ be the corresponding unipotent class.  Assume that $\cF \in \Db_G(\cU)$ is $!$-even and supported on $\overline{C_\phy}$, and that the following condition holds:
\begin{equation}\label{eqn:!-good}
\parbox{4in}{For every unipotent class $j_C: C \hookrightarrow \cU$ and every irreducible local system $E$ occurring in some cohomology sheaf $\mathcal{H}^k(j_C^!\cF)$, the simple perverse sheaf $\IC(C,E)$ occurs as a direct summand of $\cK$.}
\end{equation}
Then  $\cF$ lies in the triangulated subcategory $\Db_G(\cU)_\cK \subset \Db_G(\cU)$ generated by the direct summands of $\cK$, and $\cS(\cF)$ lies in $\Db(\SW)_{\preceq \phy}$.
\end{lem}
\begin{proof}
Let us say that an object in $\Db_G(\cU)$ is \emph{$!$-good} if condition~\eqref{eqn:!-good} holds for it.  We proceed by induction on the number of unipotent classes in the support of $\cF$.  Choose a class $C$ that is open in the support of $\cF$, and let $Z = \supp \cF \smallsetminus C$.  Let $i: Z \to \cU$ be the inclusion map, and form the distinguished triangle
\begin{equation}\label{eqn:dguk-dt1}
i_*i^!\cF \to \cF \to j_{C*}j_C^*\cF \to.
\end{equation}
The first term is clearly $!$-even and $!$-good, so by induction, it lies in $\Db_G(\cU)_\cK$, and $\cS(i_*i^!\cF) \in \Db(\SW)_{\preceq \phy}$.  Since applying $\cS$ to~\eqref{eqn:dguk-dt1} yields a short exact sequence, it suffices to prove that the conclusions of the lemma hold for $j_{C*}j_C^*\cF$.

As explained in~\cite[\S 2]{jmw}, the fact that $\cF$ is $!$-even implies that $j_C^*\cF \in \Db_G(C)$ is isomorphic to the direct sum of objects of the form $E[2k]$, where $E$ is an irreducible local system.  We may as well assume that $j_C^*\cF \cong E$ for some such $E$.  Moreover, since $\cF$ is $!$-good, $\IC(C,E)$ must occur in $\cK$, say as $\IC(C,E) \cong \IC_\chi$.

Consider the distinguished triangle
\begin{equation}\label{eqn:dguk-dt2}
i_*i^!\IC(C,E) \to \IC_\chi \to j_{C*}E[\dim C] \to.
\end{equation}
By~\cite[Theorem~24.8]{lus:cs5}, $\IC_\chi$ is both $!$-even and $!$-good.  (More precisely, that result asserts that each summand of $\cK$ is $*$-even and ``$*$-good''; we obtain the required facts by applying it to the Verdier dual of $\cK$.)  Therefore, the first term in~\eqref{eqn:dguk-dt2} is $!$-even and $!$-good as well, so the conclusions of the lemma hold for it by induction.  Those conclusions obviously hold for $\IC_\chi$, so they also hold for $j_{C*}E$, as desired.
\end{proof}

\begin{lem}\label{lem:tk-nabla}
We have $\nabla_\chi \in \SW\lgmod$.
\end{lem}
\begin{proof}
Suppose $\IC_\chi \cong \IC(C,E)$.  Let $i : \overline{C} \smallsetminus C \hookrightarrow \cU$ be the inclusion map.  Recall, as in the proof of Lemma~\ref{lem:dguk}, that $i_*i^!\IC(C,E)$ is $!$-even and satisfies~\eqref{eqn:!-good}.  Let $N = \cS(i_*i^!\IC(C,E))$ and $K = \cS(j_{C*}E[\dim C])$.  The distinguished triangle $i_*i^!\IC(C,E) \to \IC_\chi \to j_{C*}E[\dim C] \to$ gives rise to a short exact sequence
\[
0 \to N \to P_\chi \to K \to 0.
\]
By Lemma~\ref{lem:dguk}, $N$ lies in $\Db(\SW)_{\prec [\chi]}$.  We claim that $K \cong \nabla_\chi$.  By the uniqueness asserted in Proposition~\ref{prop:nabla-exist}, it suffices to check that $\Hom(M,K) = 0$ for $M \in \Dm(\SW)_{\prec [\chi]}$ or $\Dp(\SW)_{\prec [\chi]}$.  Since $K$ is a bounded complex, we may restrict our attention to $M \in \Db(\SW)_{\prec [\chi]}$, and indeed to $M$ of the form $P_\psi\la n\ra$ with $\psi \prec \chi$.  Since $K$ lies in $\SW\lgmod$ and $P_\psi$ is projective, we have that $\uHom^i(P_\psi, K) = 0$ for $i \ne 0$.  For $i = 0$, we see from~\eqref{eqn:cs-ff} that $\gr_k \Hom(P_\psi, K) \cong \Hom^{2k}(\IC_\psi, j_{C*}E[\dim C]) = 0$.
\end{proof}

Referring to the construction of $\nabla_\chi$ in Proposition~\ref{prop:nabla-exist}, it can be seen that when it lies in $\SW\lgmod$, it admits the following explicit description:
\begin{equation}\label{eqn:tk-formula}
\nabla_\chi =
P_\chi / \Bigg(\sum_{\substack{g \in \Hom(P_\psi\la n\ra, P_\chi)\\
\psi \prec \chi, \ n > 0}} \im g\Bigg).
\end{equation}
In~\cite{kat:hsgp}, this module is denoted $\tK_\chi$ .  We may also form the quotient
\begin{equation}\label{eqn:bnabla-defn}
\bnabla_\chi =
\nabla_\chi / \Bigg(\sum_{\substack{g \in \Hom(\nabla_\psi\la n\ra, \nabla_\chi)\\
\psi \sim \chi, \ n > 0}} \im g\Bigg)
\cong
P_\chi / \Bigg(\sum_{\substack{g \in \Hom(P_\psi\la n\ra, P_\chi)\\
\psi \precsim \chi, \ n > 0}} \im g\Bigg).
\end{equation}
Modules of this form are called \emph{traces} in~\cite{kat:hsgp} and are denoted by $K_\chi$ or $K_\chi^{\mathbf{c}}$.

The following theorem summarizes the properties of the $\nabla_\chi$ and the $\bnabla_\chi$ in this situation.  Part~\eqref{it:tk-nabla} was contained in Lemma~\ref{lem:tk-nabla}, and parts~\eqref{it:nabla-filt} and~\eqref{it:free-filt} are restatements of~\cite[Corollary~3.6]{kat:hsgp} and~\cite[Theorem~4.1]{kat:asea}, respectively.  (The latter result is stated in the case where only trivial local systems arise.  However, it is straightforward to adapt Kato's arguments to drop this assumption.)

\begin{thm}[Kato]\label{thm:kato}
Assume that $\precsim$ arises from a generalized Springer correspondence.  Then we have:
\begin{enumerate}
\item Each $\nabla_\chi$ lies in $\SW\lgmod$.\label{it:tk-nabla}
\item Each $\nabla_\chi$ admits a filtration whose subquotients are of the form $\bnabla_\psi\la n\ra$ with $\psi \sim \chi$.\label{it:nabla-filt}
\item Each $P_\chi$ admits a filtration whose subquotients are of the form $\nabla_\psi\la n\ra$ with $\psi \precsim \chi$.\label{it:free-filt}
\end{enumerate}
\end{thm}

\section{Module categories associated to a phylum}
\label{sect:modcat}

For the remainder of the paper, we will treat Theorem~\ref{thm:kato} as a ``black box.''  To be more precise, the proofs in this section and the next avoid geometric arguments, and are written so as to be able to accommodate arbitrary complex reflection groups and arbitrary preorders.  Since the proofs make use of Theorem~\ref{thm:kato}, the results below are, for the moment, only known to hold when $W$ and $\precsim$ come from a generalized Springer correspondence.  However, if, in the future, Theorem~\ref{thm:kato} is shown to hold for other $W$ and $\precsim$, then the results below would automatically hold in those new cases as well.

In this section, we study various special classes of modules associated to a phylum, as well as certain related abelian and triangulated categories.  We will require the following notions.

\begin{defn}\label{defn:phy-good}
Let $\phy$ be a phylum.  An object of $\SW\lgmod$ is said to be:
\begin{itemize}
\item \emph{$\phy$-good} if it admits a (possibly infinite) filtration whose subquotients are various $\bnabla_\chi\la n\ra$ with $\chi \in \phy$.
\item \emph{$\phy$-quasicostandard} if it is $\phy$-good and finite-dimensional.
\item \emph{$\phy$-projective} if it is a direct sum of various $\nabla_\chi\la n\ra$ with $\chi \in \phy$.
\item \emph{$\phy$-presentable} if it is the cokernel of a map between $\phy$-projective modules.
\end{itemize}
\end{defn}

The following lemma tells us that ``$\phy$-quasicostandard'' is not an empty concept.

\begin{lem}[{Kato~\cite[Lemma~2.15]{kat:hsgp}}]\label{lem:bnabla-findim}
Each $\bnabla_\chi$ is a finite-dimensional $\SW$-module.  As a $W$-representation, $\bnabla_\chi$ contains a copy of $L_\chi$ with multiplicity~$1$ and various other $L_\theta\la m\ra$ with $\theta \succ \chi$ and $m > 0$.
\end{lem}
\begin{proof}
This is immediate from the definition of the $\bnabla_\chi$.
\end{proof}

The term ``$\phy$-projective'' is justified by the fact that such a module is a projective object in the following Serre subcategory of $\SW\lgmod$ (see~\cite[Corollary~3.8]{kat:hsgp}):
\[
\SW\lgmod_{\not\prec \phy} = \left\{ M \in \SW\lgmod \;\Big|\, 
\begin{array}{c}
\text{$M$ contains no $W$-stable subspace} \\
\text{isomorphic to $L_\theta\la m\ra$ if $[\theta] \prec f$}
\end{array} \right\}.
\]
We will also study the additive categories
\begin{align*}
\cQ_{\phy} &= \{ M \in \SW\lgmod \mid \text{$M$ is $\phy$-quasicostandard} \}, \\
\cP_{\phy} &= \{ M \in \SW\lgmod \mid \text{$M$ is $\phy$-presentable} \},
\end{align*}
as well as the triangulated categories
\begin{gather*}
\Dfd(\SW)_{\preceq \phy} = \Db(\SW)_{\preceq \phy} \cap \Dfd(\SW), \\\Dfd(\SW)_{\phy} = \Db(\SW)_{\phy} \cap \Dfd(\SW), \qquad
\Dfd(\SW)^{\phy} = \Db(\SW)^{\phy} \cap \Dfd(\SW). 
\end{gather*}

\begin{lem}\label{lem:nabla-filt-ses}
For any $\chi \in \Irr(W)$ and any $n \ge 0$, there is a short exact sequence
\[
0 \to A_n \to \nabla_\chi \to Y_n \to 0
\]
where $A_n$ is $[\chi]$-good with grades${}\ge n$, and $Y_n$ is $[\chi]$-quasicostandard.  In particular, $\bnabla_\chi$ is a quotient of $Y_n$, and for any object $M \in \Db(\SW)$, we have
\[
\RHom(M, \nabla_\chi) \cong \RHom(M, Y_n) \qquad\text{for $n \gg 0$.}
\]
\end{lem}
\begin{proof}
The first part of this lemma is just a restatement of Theorem~\ref{thm:kato}\eqref{it:nabla-filt}, together with the observation that for fixed $k$, only finitely many subquotients of $\nabla_\chi$ can have the form $\bnabla_\psi\la k\ra$.  Next, given $M \in \Db(\SW)$, choose some bounded complex of projectives that represents $M$, and let $n$ be large enough that each term of that complex is generated in grades${}< n$.  There is no nonzero morphism from such a complex to any object with grades${}\ge n$.  In particular, $\RHom(M, A_n) = 0$, and the last assertion follows.
\end{proof}

\begin{lem}\label{lem:phy-ab}
For any phylum $\phy$, $\cP_\phy$ is an abelian category with enough projectives.  The projective objects in $\cP_\phy$ are precisely the $\phy$-projective $\SW$-modules, and the simple objects in $\cP_{\phy}$ are the $\bnabla_\chi\la n\ra$ with $\chi \in \phy$.

Moreover, an $\SW$-module $M$ is $\phy$-quasicostandard if and only if it is $\phy$-present\-able and finite-dimensional.  In particular, $\cQ_{\phy}$ is a Serre subcategory of $\cP_{\phy}$.
\end{lem}
\begin{proof}
We proceed in several steps.  The first two take place in $\SW\lgmod_{\not\prec \phy}$, and the later ones in $\cP_{\phy}$.

{\it Step 1. The $\bnabla_\chi$ are $\phy$-presentable.}  Let $M$ be the kernel of the obvious map $\nabla_\chi \to \bnabla_\chi$.  From Theorem~\ref{thm:kato}\eqref{it:nabla-filt} or Lemma~\ref{lem:nabla-filt-ses}, we know that $M$ is $\phy$-good.  In particular, $M$ is generated by its subspaces that are isomorphic (as $W$-representations) to various $L_\psi\la n\ra$ with $\psi \in \phy$.  Indeed, $M$ is generated by a finite number of such subspaces, so $M$ is a quotient of some $\phy$-projective module.  The claim follows.

{\it Step 2. The class of $\phy$-presentable modules is stable under extensions.}  Let $0 \to A \to B \to C \to 0$ be a short exact sequence in $\SW\lgmod$.  We wish to show that if $A$ and $C$ are $\phy$-presentable, then $B$ is as well.  In fact, a nine-lemma argument shows that it suffices to prove the following weaker statement: if $A$ and $C$ are quotients of $\phy$-projective modules, then $B$ is as well.  This latter statement is immediate from the observation that $A$ and $C$ necessarily lie in $\SW\lgmod_{\not\prec\phy}$, and so $B$ does as well.  (Note, however, that not every quotient of an $\phy$-projective module is $\phy$-presentable.)

{\it Step 3. $\cP_{\phy}$ is an abelian category with projectives as described above.}  Let $F = \bigoplus_{\psi \in  \phy} \nabla_\psi$, and consider the graded ring $\Gamma = \uHom(F,F)^\op$.  We have a functor $e: \cP_{\phy} \to \Gamma\lgmod$ given by $e(M) = \uHom(F,M)$.  A variation of~\cite[Propositions~2.1 and~2.5]{ars:rtaa} shows that $e$ is an equivalence of categories that takes $\phy$-projective modules to projective $\Gamma$-modules.  In particular, $\cP_{\phy}$ is naturally an abelian category with enough projectives.

{\it Step 4. The simple objects in $\cP_{\phy}$ are precisely the $\bnabla_\chi\la n\ra$.}  Abstractly, the isomorphism classes of simple objects in $\cP_{\phy}$ are in bijection with those of the indecomposable projectives.  Let $\Sigma_{\chi,n} \in \cP_{\phy}$ be the unique simple quotient of $\nabla_\chi\la n\ra$.  This object is characterized by the property that
\[
\text{for $\psi \in \phy$,} \qquad
\Hom(\nabla_\psi\la m\ra, \Sigma_{\chi,n}) = 
\begin{cases}
\C & \text{if $\psi = \chi$ and $m = n$,} \\
0 & \text{otherwise}
\end{cases} 
\]
But $\bnabla_\chi\la n\ra$ lies in $\cP_{\phy}$ and also has this property.  (This is a special case of Lemma~\ref{lem:nab-bnab-eq} below.)  We conclude that $\Sigma_{\chi,n} \cong \bnabla_\chi\la n\ra$.

{\it Step 5.  Characterization of $\phy$-quasicostandard modules.}  It follows from Step~2 that every $\phy$-quasicostandard module is $\phy$-presentable, and they are obviously finite-dimensional.  Conversely, a finite-dimensional $\phy$-presentable module must have finite length as an object of $\cP_{\phy}$.  From our description of simple objects therein, we see that such a module must be $\phy$-quasicostandard.
\end{proof}

\begin{rmk}\label{rmk:phy-fd}
Under the equivalence $e: \cP_{\phy} \simto \Gamma\lgmod$ (with the notation of the preceding proof), the category $\cQ_{\phy}$ corresponds to the category of finite-dimensional $\Gamma$-modules.
\end{rmk}

Of course, any complex of $\phy$-presentable or $\phy$-quasicostandard modules can be regarded simply as a complex of $\SW$-modules, so there are obvious functors
\[
\rho: \Db\cP_{\phy} \to \Db(\SW) \qquad\text{and}\qquad \rho: \Db\cQ_{\phy} \to \Db(\SW).
\]

\begin{prop}\label{prop:phy-ff}
The functor $\rho: \Db\cP_{\phy} \to \Db(\SW)$ is fully faithful.
\end{prop}
\begin{proof}
By arguing as in~\cite[Proposition~3.1.16]{bbd}, we can reduce this to showing that the contravariant $\delta$-functor $\Hom^i_{\Db(\SW)}({-},B)$ (for fixed $B \in \cP_{\phy}$) is effaceable.   Recall that this means that for any $A \in \cP_{\phy}$, we must show that there is a surjective map $M \to A$ such that the induced map
\[
\Hom^i(A,B) \to \Hom^i(M,B)
\]
vanishes.  Since $\cP_{\phy}$ has enough projectives, it suffices to show that
\begin{equation}\label{eqn:phy-eff}
\uHom^i_{\Db(\SW)}(P,B) = 0
\quad\text{if $P$ is $\phy$-projective, $B$ is $\phy$-presentable, and $i > 0$.}
\end{equation}
To prove this, let $n$ be the projective dimension of $P$ as an $\SW$-module.  (Here, we are using the fact that $\SW$ has finite global dimension.)  Choose an $\phy$-projective resolution $Q_\bullet$ for $B$.  Let $R$ be the complex obtained by omitting the terms $Q_i$ for $i > n$, and let $K$ be the kernel of the map $Q_n \to Q_{n-1}$.  Then there is a distinguished triangle
\[
K[n] \to R \to B \to 
\]
in $\Db(\SW)$.  By assumption, $\uHom^i(P,K[n]) = \uHom^i(P,K[n+1]) = 0$ for all $i > 0$, and Corollary~\ref{cor:nabla-ext} implies that $\uHom^i(P,R) = 0$.  Thus,~\eqref{eqn:phy-eff} holds.
\end{proof}

\begin{cor}\label{cor:phyqco-ff}
The functor $\rho: \Db\cQ_{\phy} \to \Db(\SW)$ is fully faithful.
\end{cor}
\begin{proof}
It clearly suffices to show that $\Db\cQ_{\phy} \to \Db\cP_{\phy}$ is fully faithful.  In view of Remark~\ref{rmk:phy-fd}, this follows from Lemma~\ref{lem:fd}.
\end{proof} 

\begin{prop}\label{prop:bnabla-res}
The category $\cP_{\phy}$ has finite global dimension.  In particular, each $\bnabla_\chi$ admits a finite resolution of the form
\[
0 \to Q_n \to \cdots \to Q_2 \to Q_1 \to Q_0 \to \bnabla_\chi \to 0
\]
where $Q_0 = \nabla_\chi$ and each $Q_i$ for $i > 0$ is $[\chi]$-projective with grades${}>0$.
\end{prop}
\begin{proof}
This is an immediate consequence of Proposition~\ref{prop:phy-ff} and the fact that $\SW$ has finite global dimension.
\end{proof}

\begin{lem}\label{lem:bnabla-admis}
Let $\phy$ be a phylum.  For $M \in \Db(\SW)$, we have $M \in \Db(\SW)_{\prec \phy}$ if and only if $\uRHom(M, \bnabla_\chi) = 0$ for all $\chi$ with $[\chi] \succeq \phy$.
\end{lem}
\begin{proof}
If $M \in \Db(\SW)_{\prec \phy}$, then by Proposition~\ref{prop:nabla-admis}, we have $\uRHom(M,P) = 0$ for any $\phy$-projective $P$.  It follows from Proposition~\ref{prop:bnabla-res} that $\uRHom(M, \bnabla_\chi) = 0$.

Now assume that $M \notin \Db(\SW)_{\prec \phy}$.  By Proposition~\ref{prop:nabla-admis}, there is some $k \in \Z$ and some $\chi$ with $\chi \succeq \phy$ such that $\RHom(M\la k\ra, \nabla_\chi) \ne 0$.  By Lemma~\ref{lem:nabla-filt-ses}, there is a $[\chi]$-quasicostandard object $Y_n$ such that $\RHom(M\la k\ra, Y_n) \ne 0$.  But if $\uRHom(M, \bnabla_\psi) = 0$ for all $\psi \sim \chi$, it would follow that $\uRHom(M, Y_n) = 0$, a contradiction.
\end{proof}

\begin{lem}\label{lem:del-bnab}
If $\chi \not\sim \psi$, then $\uRHom(\Delta_\chi, \bnabla_\psi) = 0$.  In particular, we have $\bnabla_\psi \in \Dfd(\SW)_{\preceq [\psi]}$.
\end{lem}
\begin{proof}
The first assertion comes from Corollary~\ref{cor:delta-nabla-orth} and Proposition~\ref{prop:bnabla-res}.  The second follows either from Proposition~\ref{prop:nabla-admis}, or from Propositions~\ref{prop:nabla-gen} and~\ref{prop:bnabla-res}.
\end{proof}

\begin{lem}\label{lem:nab-bnab-eq}
If $\chi \sim \psi$, then we have
\[
\Hom^i(\Delta_\chi, \bnabla_\psi\la n\ra) \cong \uHom^i(\nabla_\chi, \bnabla_\psi\la n\ra) \cong
\begin{cases}
\C & \text{if $i = 0$, $n = 0$, and $\chi = \psi$,} \\
0 & \text{otherwise.}
\end{cases}
\]
\end{lem}
\begin{proof}
Consider the morphism $i: \Delta_\chi \to \nabla_\chi$ of Proposition~\ref{prop:delta-nabla-cone}.  It follows from Lemma~\ref{lem:bnabla-admis} that $i$ induces an isomorphism $\uRHom(\nabla_\chi, \bnabla_\psi) \simto \uRHom(\Delta_\chi, \bnabla_\psi)$.  We now focus on the former.

It is trivial that $\uHom^i(\nabla_\chi,\bnabla_\psi)$ vanishes for $i < 0$.  When $i > 0$, Corollary~\ref{cor:nabla-ext} and Proposition~\ref{prop:bnabla-res} together imply that $\uHom^i(\nabla_\chi, \bnabla_\psi) = 0$.  Finally, when $i = 0$, the result follows from the definition of $\bnabla_\psi$.
\end{proof}

\begin{prop}\label{prop:bnabla-gen}
Let $\phy$ be a phylum.
\begin{enumerate}
\item $\Dfd(\SW)_{\preceq \phy}$ is generated as a triangulated category by the $\bnabla_\chi\la n\ra$ (resp.~the objects $\D(\bnabla_\chi)\la n\ra$) with $[\chi] \preceq \phy$.\label{it:gen-le}
\item $\Dfd(\SW)_{\phy}$ is generated as a triangulated category by the $\bnabla_\chi\la n\ra$ with $\chi \in \phy$.\label{it:gen-eq}
\item $\Dfd(\SW)^{\phy}$ is generated as a triangulated category by the objects $\D(\bnabla_\chi)\la n\ra$ with $\chi \in \phy$.\label{it:gen-dual}
\end{enumerate}
\end{prop}
\begin{proof}
It is clear that parts~\eqref{it:gen-eq} and~\eqref{it:gen-dual} are equivalent.  Similarly, $\Dfd(\SW)_{\preceq \phy}$ is stable under $\D$, so it is generated by the $\bnabla_\chi\la n\ra$ if and only if it is generated by the $\D(\bnabla_\chi)\la n\ra$.  It suffices, therefore, to consider only the assertions involving the $\bnabla_\chi\la n\ra$.

Let $D_{\preceq \phy}$ (resp,~$D_{\prec \phy}$, $D_{\phy}$, $D_{\succ \phy}$) be the triangulated category generated by the $\bnabla_\chi\la n\ra$ with $[\chi] \preceq \phy$ (resp.~$[\chi] \prec \phy$, $\chi \in \phy$, $[\chi] \succ \phy$).  Lemma~\ref{lem:del-bnab} tells us that $D_{\preceq \phy} \subset \Dfd(\SW)_{\preceq \phy}$.  

It is clear from Lemma~\ref{lem:bnabla-findim} that the set of all $\bnabla_\chi\la n\ra$ generates $\Dfd(\SW)$.  Thus, $D_{\preceq \phy}$ and $D_{\succ \phy}$ together generate $\Dfd(\SW)$, and then by Lemma~\ref{lem:bnabla-admis}, we see that $(D_{\preceq \phy}, D_{\succ \phy})$ is an admissible pair.  But Lemma~\ref{lem:admissible}\eqref{it:rt-orth} tells us that $\Dfd(\SW)_{\preceq \phy} \subset D_{\preceq \phy}$, so the two categories coincide.

Part~\eqref{it:gen-le}, now proved, implies that $D_{\preceq \phy}$ is generated by $D_{\prec \phy}$ and $D_{\phy}$ together.  Combining this with Lemmas~\ref{lem:bnabla-admis} and~\ref{lem:del-bnab}, we see that $(D_{\prec \phy}, D_{\phy})$ is an admissible pair in $D_{\preceq \phy}$.  Proposition~\ref{prop:bnabla-res} implies that $D_{\phy} \subset \Dfd(\SW)_{\phy}$, but Lemma~\ref{lem:admissible}\eqref{it:lt-orth} then tells us that $\Dfd(\SW)_{\phy} \subset D_{\phy}$, so these categories coincide, as desired.
\end{proof}

\section{Main results}
\label{sect:main}

\subsection{Construction of the exotic $t$-structure}
\label{subsect:tstruc}

In this subsection and the next, we rely on the general framework developed in~\cite{a:pcsnc} for constructing quasi-hereditary $t$-structures and proving derived equivalences.  The main task is to show that the collection of objects $\{ \bnabla_\chi\la n\ra \}$ satisfy the axioms in~\cite{a:pcsnc} for a ``dualizable abelianesque graded quasi-exceptional set.'' The definition of these terms is recalled in the statements of the first two propositions below.

One caveat should be kept in mind: the arguments given in~\cite{a:pcsnc} assume that the set used to label various objects is equipped with a partial order, not merely a preorder.  Below, we will give careful statements of the preorder versions of the definitions and results we need from~\cite{a:pcsnc}.  The task of rewriting the proofs from~\cite{a:pcsnc} to accommodate preorders, however, will not be done here, as it is entirely straightforward and tedious.

\begin{prop}\label{prop:qexc}
The collection of objects $\{\bnabla_\chi\}_{\chi \in \Irr(W)}$ is a graded quasi-excep\-tion\-al set in $\Dfd(\SW)$.  In other words, we have:
\begin{enumerate}
\item If $\chi \prec \psi$, then $\uRHom(\bnabla_\chi, \bnabla_\psi) = 0$.\label{it:qexc1}
\item If $\chi \sim \psi$ and $i < 0$, then $\uHom^i(\bnabla_\chi, \bnabla_\psi) = 0$.  Moreover,\label{it:qexc2}
\[
\uHom(\bnabla_\chi, \bnabla_\psi) \cong 
\begin{cases}
\C & \text{if $\chi = \psi$,} \\
0 & \text{otherwise.}
\end{cases}
\]
\item If $\chi \sim \psi$, $i > 0$, and $n\le 0$, then $\Hom^i(\bnabla_\chi, \bnabla_\psi\la n\ra) = 0$.\label{it:qexc3}
\item The objects $\{ \bnabla_\chi\la n\ra \}$ generate $\Dfd(\SW)$ as a triangulated category.\label{it:qexc4}
\end{enumerate}
In addition, this quasi-exceptional set is abelianesque, meaning that
\begin{enumerate}\setcounter{enumi}{4}
\item If $i < 0$, then $\uHom^i(\bnabla_\chi, \bnabla_\psi) = 0$ for all $\chi, \psi$.\label{it:qexc5}
\end{enumerate}
\end{prop}
\begin{proof}
\eqref{it:qexc1}~This follows from Lemmas~\ref{lem:bnabla-admis} and~\ref{lem:del-bnab}.

\eqref{it:qexc2}~The first assertion is obvious from the fact that $\bnabla_\chi \in \SW\lgmod$.  For the second, note that $L_\chi$ is the unique simple quotient of $\bnabla_\chi$ as an $\SW$-module, and recall from Lemma~\ref{lem:bnabla-findim} that the multiplicity of $L_\chi$ as a composition factor of $\bnabla_\psi\la n\ra$ is $1$ if $\psi = \chi$ and $n = 0$, and $0$ otherwise.

\eqref{it:qexc3}~Let $M$ be the kernel of the map $\nabla_\chi \to \bnabla_\chi$.  By Lemma~\ref{lem:nab-bnab-eq}, we have $\Hom^i(\nabla_\chi, \bnabla_\psi\la n\ra) = 0$ for $i > 0$, so there is a surjective map
\[
\Hom^{i-1}(M, \bnabla_\psi\la n\ra) \to \Hom^i(\bnabla_\chi, \bnabla_\psi\la n\ra)
\]
for all $i > 0$.  (It is an isomorphism for $i > 1$.)  Now, examining Proposition~\ref{prop:bnabla-res}, we see that $M$ has a finite $\phy$-projective resolution $Q_\bullet$ where each term has grades${}>0$, so for $n \le 0$, the module $Q_j\la -n\ra$ has strictly positive grades as well.  Using Lemma~\ref{lem:nab-bnab-eq} once again, we have $\Hom(Q_j,\bnabla_\psi\la n\ra) = 0$ for all $j$, and hence $\Hom^{i-1}(M,\bnabla_\psi \la n\ra) =0$ for all $i > 0$.

\eqref{it:qexc4}~This is contained in Proposition~\ref{prop:bnabla-gen}.

\eqref{it:qexc5}~This is obvious, since the $\bnabla_\chi$ lie in $\SW\lgmod$.
\end{proof}

\begin{prop}\label{prop:dual}
The quasi-exceptional set $\{ \bnabla_\chi \}$ is dualizable.  That is, for each $\chi$, there is an object $\bDelta_\chi$ and a morphism $i: \bDelta_\chi \to \bnabla_\chi$ such that:
\begin{enumerate}
\item The cone of $i$ lies in $\Dfd(\SW)_{\prec [\chi]}$.\label{it:dual-cone}
\item If $\chi \succ \psi$, the $\uRHom(\bDelta_\chi, \bnabla_\psi) = 0$.\label{it:dual-ext}
\end{enumerate}
\end{prop}
\begin{proof}
Let $\phy = [\chi]$.  In the proof of Proposition~\ref{prop:bnabla-gen}, we saw that the categories $(\Dfd(\SW)^{\phy}, \Dfd(\SW)_{\prec \phy})$ form an admissible pair in $\Dfd(\SW)_{\preceq \phy}$. Apply Lemma~\ref{lem:admissible}\eqref{it:adm-dt} to the object $\bnabla_\chi$, we obtain a distinguished triangle
\[
\imath(\bnabla_\chi) \overset{i}{\to} \bnabla_\chi \to \jmath(\bnabla_\chi) \to
\qquad\text{with $\imath(\bnabla_\chi) \in \Dfd(\SW)^{\phy}$ and $\jmath(\bnabla_\chi) \in \Dfd(\SW)_{\prec \phy}$.}
\]
Set $\bDelta_\chi = \imath(\bnabla_\chi)$, and let $i: \bDelta_\chi \to \bnabla_\chi$ as above.  Then part~\eqref{it:dual-cone} of the proposition is clear, and part~\eqref{it:dual-ext} holds because $\bDelta_\chi \in \Db(\SW)^{\phy}$ and $\bnabla_\psi \in \Db(\SW)_{\prec \phy}$.
\end{proof}

We are at last ready to define the exotic $t$-structure.  As with the preceding propositions, a key definition---that of a ``weakly quasi-hereditary category''---is given in the body of the following theorem.

\begin{thm}\label{thm:tstruc}
The categories
\begin{align*}
\Dfd(\SW)^{\le 0} &= \{ X \in \Dfd(\SW) \mid \text{$\uHom^i(X, \bnabla_\chi) = 0$ for all $i < 0$} \}, \\
\Dfd(\SW)^{\ge 0} &= \{ X \in \Dfd(\SW) \mid \text{$\uHom^i(\bDelta_\chi, X) = 0$ for all $i < 0$} \},
\end{align*}
constitute a bounded $t$-structure on $\Dfd(\SW)$.  Its heart, denoted by
\[
\Ex = \Dfd(\SW)^{\le 0} \cap \Dfd(\SW)^{\ge 0},
\]
is a finite-length abelian category.  All $\bDelta_\chi\la n\ra$ and $\bnabla_\chi\la n\ra$ belong to $\Ex$.  The image of the natural map $\bDelta_\chi\la n\ra \to \bnabla_\chi\la n\ra$, denoted by $\Sigma_\chi\la n\ra$, is a simple object of $\Ex$, and every simple object is of this form.

Furthermore, $\Ex$ is weakly quasi-hereditary.  This means that, letting $\Ex^{\prec \phy}$ denote the Serre subcategory of $\Ex$ generated by the $\Sigma_{\psi}\la n\ra$ with $[\psi] \prec \phy$, we have:
\begin{enumerate}
\item The kernel of $\bDelta_\chi \to \Sigma_\chi$ lies in $\Ex^{\prec [\chi]}$, and if $\psi \prec \chi$, then\label{it:qhered-std}
\[
\uHom(\bDelta_\chi, \Sigma_\psi) = \uExt^1(\bDelta_\chi, \Sigma_\psi) = 0.
\]
\item The cokernel of $\Sigma_\chi \to \bnabla_\chi$ lies in $\Ex^{\prec [\chi]}$, and if $\psi \prec \chi$, then\label{it:qhered-costd}
\[
\uHom(\Sigma_\psi, \bnabla_\chi) = \uExt^1(\Sigma_\psi, \bnabla_\chi) = 0. \qed
\]
\end{enumerate}
\end{thm}
\begin{proof}
According to~\cite[Theorem~2.10]{a:pcsnc}, this is a consequence of Propositions~\ref{prop:qexc} and~\ref{prop:dual}.
\end{proof}

\begin{rmk}\label{rmk:strong-qhered}
In~\cite{a:pcsnc}, categories satisfying the conditions~\eqref{it:qhered-std} and~\eqref{it:qhered-costd} were simply called ``quasi-hereditary''; the adjective ``weak'' was not used.  That terminology is compatible with~\cite{a:qhpss, bez:qes}, but not with most other sources, such as~\cite{rin:cmgf}.  In the more common usage of ``quasi-hereditary,'' one would require that
\[
\uExt^1(\bDelta_\chi, \Sigma_\psi) = \uExt^1(\Sigma_\psi, \bnabla_\chi) =0 \qquad \text{if $\psi \precsim \chi$,}
\]
not just when $\psi \prec \chi$.  This stronger condition does not hold for $\Ex$ in general.
\end{rmk}

\subsection{Derived equivalence}
\label{subsect:dereq}

We continue to rely on the machinery that was developed in~\cite{a:pcsnc}.  The $\bDelta_\chi$ are not, in general, objects of $\SW\lgmod$, but in the context of $\Ex$, they can often be treated symmetrically with the $\bnabla_\chi$.  For instance, we can now formulate a notion dual to ``$\phy$-quasicostandard.''

\begin{defn}\label{defn:phy-dual}
Let $\phy$ be a phylum.  An object of $\Db(\SW)$ is said to be \emph{$\phy$-quasistandard} if it lies in $\Ex$ and admits a filtration whose subquotients are various $\bDelta_\chi\la n\ra$.
\end{defn}

Next, we establish statements parallel to Lemma~\ref{lem:bnabla-admis} and Proposition~\ref{prop:bnabla-gen}.

\begin{prop}\label{prop:bdelta}
Let $\phy$ be a phylum.
\begin{enumerate}
\item $\Dfd(\SW)_{\preceq \phy}$ is generated as a triangulated category by the $\bDelta_\chi\la n\ra$ with $[\chi] \preceq \phy$.\label{it:dgen-le}
\item $\Dfd(\SW)^{\phy}$ is generated as a triangulated category by the $\bDelta_\chi\la n\ra$ with $\chi \in \phy$.\label{it:dgen-eq}
\item For any object $M \in \Db(\SW)$, we have $M \in \Db(\SW)_{\prec \phy}$ if and only if $\uRHom(\bDelta_\chi, M) = 0$ for all $\chi$ with $[\chi] \succeq \phy$.\label{it:bdel-admis}
\end{enumerate}
\end{prop}
\begin{proof}
In view of Proposition~\ref{prop:dual}\eqref{it:dual-cone}, it is easy to see by induction on $\phy$ that part~\eqref{it:dgen-le} above holds.  Then, part~\eqref{it:dgen-eq} can be deduced from part~\eqref{it:dgen-le} using the same argument that was used to deduce Proposition~\ref{prop:bnabla-gen}\eqref{it:gen-eq} from Proposition~\ref{prop:bnabla-gen}\eqref{it:gen-le}.

Finally, since $\Db(\SW)_{\prec \phy}$ is stable under $\D$, it follows from Lemma~\ref{lem:bnabla-admis} that $M \in \Db(\SW)_{\prec \phy}$ if and only if $\uRHom(\D(\bnabla_\chi),M) = 0$ for all $\chi$ with $[\chi] \succeq \phy$.  Part~\eqref{it:dgen-eq} implies that the latter condition is equivalent to the one appearing in part~\eqref{it:bdel-admis} of the proposition.
\end{proof}

The next statement is immediate from Propositions~\ref{prop:dual}, \ref{prop:bnabla-gen}\eqref{it:gen-eq}, and~\ref{prop:bdelta}\eqref{it:dgen-eq}; cf.~Proposition~\ref{prop:delta-nabla-cone}.

\begin{cor}\label{cor:std-costd-equiv}
The equivalence $T_{\phy}: \Db(\SW)_{\phy} \simto \Db(\SW)^{\phy}$ of~\eqref{eqn:t-defn} restricts to an equivalence
\[
T_{\phy}: \Dfd(\SW)_{\phy} \simto \Dfd(\SW)^{\phy}.
\]
We have $T_{\phy}(\bnabla_\chi) \cong \bDelta_\chi$.  More generally, $T_{\phy}$ takes $\phy$-quasicostandard objects to $\phy$-quasistandard objects. \qed
\end{cor}

\begin{lem}\label{lem:efface}
Let $\phy$ be a phylum, and let $\chi \in \phy$.
\begin{enumerate}
\item For any morphism $g: \bnabla_\chi \to M[d]$ where $d > 0$ and $M$ is $\phy$-quasicostan\-dard, there exists an $\phy$-quasicostandard module $Y$ and a surjective map $h: Y \twoheadrightarrow \bnabla_\chi$ such that $g \circ h = 0$.\label{it:eff-costd}
\item For any morphism $g: M[d] \to \bDelta_\chi$ where $d < 0$ and $M$ is $\phy$-quasistandard, there exists an $\phy$-quasistandard module $Y$ and an injective map $h: \bDelta_\chi \hookrightarrow Y$ such that $h \circ g = 0$.\label{it:eff-std}
\end{enumerate}
\end{lem}
\begin{proof}
Let us consider the following additional statement:
\begin{enumerate}
\setcounter{enumi}{2}
\item For any morphism $g: M[d] \to \bnabla_\chi$ where $d < 0$ and $M$ is $\phy$-quasicostan\-dard, there exists an $\phy$-quasicostandard module $Y$ and an injective map $h: \bnabla_\chi \hookrightarrow Y$ such that $h \circ g = 0$.\label{it:eff-codual}
\end{enumerate}
Statements~\eqref{it:eff-costd} and~\eqref{it:eff-codual} both involve objects in the abelian category $\cQ_{\phy}$.  They both follow from the claim that the $\delta$-functor $\Hom^i_{\Dfd(\SW)}(A,B)$ (for $A, B \in \cQ_{\phy}$) is effaceable in both variables, and the latter is a consequence of Corollary~\ref{cor:phyqco-ff}.  Finally, Corollary~\ref{cor:std-costd-equiv} implies that conditions~\eqref{it:eff-std} and~\eqref{it:eff-codual} are equivalent to one another.
\end{proof}

\begin{thm}\label{thm:dereq}
There is an equivalence of triangulated categories
\[
\Db\Ex \simto \Dfd(\SW).
\]
\end{thm}
\begin{proof}
The criterion given in~\cite[Theorem~3.15]{a:pcsnc} states that for a $t$-structure arising from a dualizable abelianesque graded quasi-exceptional set, such a derived equivalence holds provided that the conditions in~\cite[Definition~3.5]{a:pcsnc} are satisfied.  That is precisely the content of Lemma~\ref{lem:efface}.
\end{proof}

\subsection{Tilting}
\label{subsect:tilting}

We conclude with a speculation about a possible alternative approach to Theorem~\ref{thm:dereq}.  It is not too difficult to deduce from Lemma~\ref{lem:nabla-filt-ses} and Theorem~\ref{thm:kato}\eqref{it:free-filt} that each $P_\chi$ is a projective limit of objects with a ``costandard filtration.''  To be more precise, one can show that there is a sequence of surjective maps
\[
\cdots \to M_2 \to M_1 \to M_0 \to 0
\]
in $\Ex$ such that the kernel of each map $M_i \to M_{i-1}$ admits a filtration whose subquotients are various $\bnabla_\chi\la k\ra$, and such that
\[
P_\chi \cong \lim_{\leftarrow} M_i.
\]
With a bit more effort, one can show that the corresponding statement with $\bDelta_\chi\la k\ra$ also holds.  (The latter requires more effort because the $\bDelta_\chi\la k\ra$ do not, in general, lie in $\SW\lgmod$ or in any other obvious $t$-structure on $\Db(\SW)$, so one does not have the luxury of studying the limit of the $M_i$ inside an abelian category.)  Since $P_\chi$ is a projective limit in both ways, one might say that $P_\chi$ is a ``protilting'' object for $\Ex$.

However, this viewpoint is somewhat unsatisfactory.  In quasi-hereditary categories (in the usual sense, cf.~Remark~\ref{rmk:strong-qhered}), where tilting objects give rise to derived equivalences, a key role is played by the fact that tilting objects have no self-extensions.  That follows from the fact that a standard object can have no extensions by a costandard object.  In contrast, in our setting, although it is true that the $P_\chi$ have no self-extensions, this cannot readily be deduced from the fact that they are protilting, because it can happen that $\Ext^1(\bDelta_\chi, \bnabla_\psi) \ne 0$ if $\chi \sim \psi$.

In other words, the property of being protilting does not seem to have any useful consequences.  A related observation is that the objects $\Delta_\chi$ and $\nabla_\chi$, which have better $\Ext^1$-vanishing properties, have no role in the notion of ``protilting,'' nor in the proof of Theorem~\ref{thm:dereq}.

A possible framework for remedying this situation is that of ``properly stratified categories,'' which have been studied by Frisk--Mazorchuk~\cite{fm:psat}. These are weakly quasi-hereditary categories equipped with additional classes of objects with good $\Ext^1$-vanishing properties.  In this paper, the notation for the objects $\Delta_\chi$, $\nabla_\chi$, $\bDelta_\chi$, $\bnabla_\chi$ was chosen to be reminiscent of theirs.

The category $\Ex$ is not a properly stratified category because the $\Delta_\chi$ and $\nabla_\chi$ do not belong to $\Ex$.  But perhaps it would be possible to develop a ``pro-'' version of the Frisk--Mazorchuk theory, one whose axioms are satisfied by $\Ex$.  In such a theory, Theorem~\ref{thm:dereq} might simply be a special case of a general Ringel duality result, analogous to~\cite[Theorem~5]{fm:psat}. 


\end{document}